\documentclass[12pt]{amsart}

\usepackage{graphicx,xcolor}
\usepackage{amsmath,amsthm,amssymb}
\newtheorem{theorem}{Theorem}
\newtheorem{lemma}{Lemma}
\newcommand{\bD}{\mathbb{D}}
\newcommand{\abs}[1]{\left\lvert #1 \right\rvert}

\begin{document}

\title[Extr. problems for trinomials with fold symmetry]
{Some extremal problems for trinomials with fold symmetry}

\author{Dmitriy Dmitrishin} 
\address{Department of Applied Mathematics, Odessa National Polytechnic University,  Odessa, 65044, Ukraine}
\email{dmitrishin@op.edu.ua }

\author{Daniel Gray}
\address{Department of Mathematics,  Georgia Southern University, Statesboro GA, 30460, USA}

\author{Alex Stokolos}
\address{Department of Mathematics,  Georgia Southern University, Statesboro GA, 30460, USA}


\begin{abstract}
The famous T. Suffridge polynomials have many extremal properties: the maximality of coefficients when the leading 
coefficient is maximal; the zeros of the derivative are located on the unit circle; the maximum radius of stretching the unit 
disk with the schlicht normalization $F(0)=0$, $F'(0)=1$; the maximum size of the unit disk contraction in the direction 
of the real axis for univalent polynomials with the normalization $F(0)=0$, $F(1)=1.$ However, under the standard 
symmetrization method $\sqrt[T]{F(z^T)}$, these polynomials become functions which are not polynomials. How can we 
construct the polynomials with fold symmetry that have properties similar to those of the Suffridge polynomial? What
values will the corresponding extremal quantities take in the above-mentioned extremal problems? The paper is devoted 
to solving these questions for the case of the trinomials $F(z)=z+az^{1+T}+bz^{1+2T}$. Also, there are
suggested hypotheses for the general case in the work.

\end{abstract}



\maketitle

\section{Introduction}

This work is motivated by problems of classical geometric complex analysis, which studies various extremal properties
of the univalent (or similar in properties to univalent) in the central unit disk $\bD=\{z\in\mathbb{C}: \abs{z}<1\}$ functions
$F(z)$ with different normalizations. 

One of the first fundamental works in this theory was the 1916 paper by Ludwig Bieberbach, in which he proved the 
exact estimate for the second coefficient of the univalent in $\bD$ functions of the form
\begin{equation}\label{F}
F(z)=z+\sum_{j=2}^{\infty}a_j z^j,
\end{equation}
namely, $\abs{a_2} \le2$. This estimate immediately implies the famous Koebe $1/4$ theorem. In that same work, Bieberbach
stated in a footnote that in the general case there must hold $\abs{a_j} \le j$. This innocently-looking statement became 
known as the Bieberbach conjecture, which, since then, has been the engine of the development of geometric complex
analysis for 105 years. In 1984, the Bieberbach conjecture was proved by Louis de Branges, who also showed that the
extremal function, up to rotation, is the Koebe function
$$
K(z)=\frac{z}{(1-z)^2}=z+2z^2+3z^3+\ldots\,.  
$$

Traditionally, the class of univalent in $\bD$ functions of the form \eqref{F} is denoted by the symbol $\mathcal S$ (from the 
German word Schlicht). For each function $F\in\mathcal S$, the function $H(z)=\sqrt[T]{F(z^T)}\in\mathcal S$ ($z\in\bD$,
$T=1,2,\ldots$) and maps the unit disk to a domain with $T$-fold symmetry ($T$-fold symmetric function). 
G. Szeg\H{o} suggested that for the coefficients of the $T$-symmetric univalent function
$$
F(z)=z+\sum_{j=2}^{\infty}a_j z^{1+(j-1)T},
$$
$\max{\abs{a_j}}\le O\left( (1+(j-1)T)^{-1+2/T} \right)$ \cite{Bas,Lew}. However, J. Littlewood showed \cite{Litt} that
this hypothesis is false for sufficiently large $T$. The results of N. Makarov \cite{Mak} imply that the growth rate of 
coefficients of $T$-symmetric univalent functions equals $2/T-1$ for $T\le2/B(1)$ and $B(1)-1$ for $T\ge2/B(1)$, where
$B(t)$ is the universal integral means spectrum for the class of univalent functions. This theorem, together with the 
important paper by L. Carleson and P. Jones \cite{CJ}, suggest that G. Szeg\H{o}'s conjecture holds for $T\le8$ and
fails for $T\ge9$. D. Beliaev and S. Smirnov \cite{BS2005} proved that the conjecture is indeed wrong for $T\ge9$. The
general problem of finding the correct estimates is still wide open even for bounded functions, see, e.g., \cite{BS2010}. 
Thus, the transfer of the known extremal properties of general functions, belonging to the class $\mathcal S$, to $T$-symmetric
functions of the class $S$ is not, generally, a trivial task.

Even more significant difficulties arise when trying to establish the well-known properties of univalent polynomials
for $T$-symmetric univalent polynomials. 

One of the most famous families of univalent polynomials solving various extremal problems is the T. Suffridge 
polynomials \cite{Suf}:
$$
S_{k,N}(z)=z+\sum_{j=2}^{N} \sigma_{j,k}z^{j},
$$
$$
\sigma_{j,k}=\left(1-\frac{j-1}N\right)\frac{\sin{\frac{\pi kj}{N+1}}}{\sin{\frac{\pi k}{N+1}}},\;
k=1,2,\ldots,N.
$$
Note that Suffridge polynomials are the matter of greater part of the section on univalent polynomials in the 
well-known book on univalent functions \cite{Dur}. 

It is easy to see that the coefficient $\sigma_{N,k}=\frac{(-1)^k}N$ possesses the extremal property, since 
the absolute value of the leading coefficient of a univalent polynomial, whose first coefficient is equal to one, can not
exceed the value of $1/N$. 

T. Suffridge proved that all the polynomials $S_{k,N}(z)$ are univalent in $\bD$. Moreover, the zeros of the derivatives
of polynomials $S_{k,N}(z)$ lie on the unit circle, and the image of the unit circle contains $N-1$ cusps \cite{Cor}. Thus,
the Suffridge polynomials are quasi-extremal polynomials in the sense of Ruscheweyh \cite{GRS}. 

Also, T. Suffridge showed that the polynomial $S(z):=S_{1,N}(z)$ is extremal, in the sense that
it maximizes the absolute value of each coefficient of any univalent polynomial of order $N$ with the leading coefficient
$1/N$.

In \cite{Br}, one more extremal property of the polynomial $S(z)$ was established, namely there was solved the
extremal problem of estimating the maximum value of the modulus of univalent polynomials with real coefficients 
of the form $P(z)=z+\sum\limits_{j=2}^N a_j z^j$ in the disk $\overline{\bD}=\{z\in\mathbb{C}:\abs{z} \le1\}$:
\begin{equation}\label{extr}
\max_{P\in\mathcal S}\left\{ \max_{z\in\overline{\bD}} \{\abs{P(z)}\}\right\}=S(1)=\frac14 \left(1+\frac1N\right)
\csc^2{\frac{\pi}{2N+2}}.
\end{equation}
Note that formula \eqref{extr} is (46) in \cite{Br}, and that the uniqueness of the corresponding extremal polynomial is not discussed.

The following extremal property of the Suffridge polynomial is related to the problem of covering intervals and is 
achieved on the class of the univalent in $\bD$ polynomials with real coefficients $P(z)$ with the normalization $P(1)=1$.
Namely,
\begin{equation}\label{min}
\min_{P}\{\abs{P(-1)}\}=\frac{\abs{S(-1)}}{S(1)}=\tan^2{\frac{\pi}{2(N+1)}}.
\end{equation} 
The polynomial $\frac1{S(1)}S(z)$ is again the only extremal polynomial in this problem \cite{Dim}. 

Note that the Suffridge polynomials are used to approximate the Koebe function \cite{Dim}, some unexpected 
applications of extremal polynomials to the solution of control problems in nonlinear discrete systems are also 
found \cite{DHKKS}.

Obviously, the $T$-symmetric function $\sqrt[T]{S(z^T)}$ is not a polynomial. What polynomial with circlurar 
symmetry will have the properties analogous to those of the Suffridge polynomial? What values will the corresponding
extremal quantities take in the considered above extremal problems? This paper is devoted to a discussion of these issues.
Let us note only that even for $T=2$, that is, in the case of odd polynomials, these problems are not easy.

The paper is organized as follows. In Section 2, polynomials with fold symmetry are presented as candidates for 
an analogue of the Suffridge polynomials. A number of hypotheses for the extremal properties of these polynomials,
which are inherent in Suffridge polynomials, are put forward. The third section gives a description of the domain of
univalence for trinomials with fold symmetry in the coefficient plane, using five parametric equations of the boundary. 
There are formulated constrained extremum problems for some functions of two variables, which are test ones for
checking the stated hypotheses. The fourth section is auxiliary, two technical Lemmas are proved there. Section 5 contains
the main Theorem of the paper, from which it follows that all the formulated hypotheses are true for trinomials with
fold symmetry. The last section is devoted to a short discussion of the obtained results.

\section{Problem statement}

In \cite{DST,DSST}, there was introduced a new class of $T$-symmetric polynomials of degree $N=1+(n-1)T$
\begin{equation}\label{Ts}
S^{(T)}(z)=z+\sum_{j=2}^{n} \left(1-\frac{(j-1)T}{1+(n-1)T}\right)
\prod_{k=1}^{j-1}\frac{\sin{\frac{\pi(2+T(k-1))}{2+T(n-1)}}}{\sin{\frac{\pi Tk}{2+T(n-1)}}}
z^{T(j-1)+1}.
\end{equation}
It is easy to check that $S^{(1)}(z)=S(z)$, $S^{(2)}(z)=-iS_{n,N}(iz)$, where $N=2n-1$, $i^2=-1$. Hence, 
the polynomial $S^{(2)}(z)$ is univalent. Note that the leading coefficient of the polynomial $S^{(T)}(z)$ has 
the maximum possible value for univalent polynomials of degree $N=1+(n-1)T$ with the first coefficient equal to one:
$\frac1{1+(n-1)T}$.

The works \cite{DSST} \cite{DST} presented the results of some numerical experiments, which, together with the  
established results, made it possible to propose a number of hypotheses. Let us consider them.

Denote the class of univalent $T$-symmetric polynomials with real coefficients of degree $N=1+(n-1)T$ and with the
first coefficient equal to one by $\mathcal S_N^T$. Hypotheses:

{\it a)} all the polynomials $S^{(T)}(z)$ are univalent in $\bD$;

{\it b)} all the polynomials $S^{(T)}(z)$ are quasi-extremal in the sense of Ruscheweyh;

{\it c)} all the polynomials $S^{(T)}(z)$ are extremal, in the sense that they maximize the absolute value of each
coefficient of any polynomial from $\mathcal S_N^T$ with the leading coefficient $1/N$;

{\it d)} $\max\limits_{P\in\mathcal S_N^T}\left\{\max\limits_{z\in\overline{\bD}} \{\abs{P(z)}\}\right\}=S^{(T)}(1)$;

{\it e)} $\min\limits_{P\in\mathcal S_N^T}\left\{ \frac1{P(1)}\abs{P\left(e^{i\frac{\pi}T}\right)} \right\}=
\frac1{S^{(T)}(1)}\abs{S^{(T)}\left(e^{i\frac{\pi}T}\right)}$.

Problem: check the truth of the formulated hypotheses for all $T=1,2,\ldots$ and $n=3$. Note that in contrast with the 
cases $n=1,2$, the problem of determining the extremal properties of polynomials from different classes even for small
$n\ge3$ is usually far from trivial \cite{Ign}. Let us only note that the problem {\it e)} for $T=2$ is solved in \cite{DHKKS},
where it is established that
$$
\min_{P\in\mathcal S_N^2}\left\{ \frac1{P(1)}\abs{P\left(e^{i\frac{\pi}2}\right)}\right\}=
\frac1{S^{(2)}(1)}\abs{S^{(2)}\left(e^{i\frac{\pi}2}\right)} =
\abs{\frac{S_{n,N}(-1)}{S_{n,N}(i)}} =\frac1n,
$$
and the extremal polynomial is unique.

\section{Domain of univalence of $T$-symmetric trinomials in the coefficient space}

In \cite{Sch}, there was considered the problem of constructing the domain of univalence for the trinomials
$$
F(z)=z+az^k+bz^m
$$
with complex coefficients. Let us present the result of this work for the case of real coefficients. So, let
$$
U_{k,m}=\{(a,b)\in\mathbb{R}^2: z+az^k+bz^m \textit{ is univalent in }\bD\}.
$$ 
Let $\frac{k-1}{m-k}=\frac{p}{q}$ be the fraction after reduction. Define the following five curves in the plane:
\begin{gather*}
\Gamma_1=\left\{(x,y):\,x=\frac{-\sin{(m-1)t}}{k\sin{(m-k)t}},\,
y=\frac{\sin{(k-1)t}}{m\sin{(m-k)t}},\,t\in\mathbb{R}\right\},\\
\Gamma_2^s=\{(x,y):\, kx+(-1)^{qs}my=(-1)^{ps+1}\},\;s=0,1,\\
\Gamma_3^s=\left\{(x,y):\, x=(-1)^{ps+1}\frac{\cos{t}}{\cos{kt}}
\frac{\tan{mt}-m\tan{t}}{k\tan{mt}-m\tan{kt}},\right.\\
\left. y=(-1)^{(p+q)s}\frac{\cos{t}}{\cos{mt}}\frac{\tan{kt}-k\tan{t}}{k\tan{mt}-m\tan{kt}},
\,t\in\mathbb{R}\right\},\;s=0,1.
\end{gather*}

It is noted further that these curves define the area $U_{k,m}$ in such a way that the boundary of this area is
contained in the union of these five curves. The set $U_{k,m}$ is a component bounded by these curves, which 
contains zero.

A thorough analysis of the equations for the boundaries of the area $U_{k,m}$ shows that these equations can be
simplified, and we also can identify the exact intervals of variation for the parameters in the parametric equations for 
setting the boundaries. Let us present these equations for the case $k=1+T$, $m=1+2T$.

We have $\frac{k-1}{m-k}=1$, $p=1$, and $q=1$. Hence,
\begin{gather*}
\frac{\cos{t}}{\cos{(1+T)t}} \frac{\tan{(1+2T)t}-(1+2T)\tan{t}}{(1+T)\tan{(1+2T)t}-(1+2T)\tan{(1+T)t}}=\\
=\frac{2T\sin{(2+2T)t}-(2+2T)\sin{2Tt}}{T\sin{(2+3T)t}-(2+3T)\sin{Tt}},\\
\frac{\cos{t}}{\cos{(1+2T)t}} \frac{\tan{(1+T)t}-(1+T)\tan{t}}{(1+T)\tan{(1+2T)t}-(1+2T)\tan{(1+T)t}}=\\
=\frac{T\sin{(2+T)t}-(2+T)\sin{Tt}}{T\sin{(2+3T)t}-(2+3T)\sin{Tt}}.
\end{gather*} 
Then the domain $U_T$ of univalence of the trinomial $F(z)=z+az^{1+T}+bz^{1+2T}$ is bounded by the curves:
\begin{gather*}
\Gamma_1=\left\{(x,y):\, x=t,\,y=\frac1{1+2T},\right.\\
\left.t\in\left[ -2\frac{1+T}{1+2T}\sin{\frac{\pi}{2+2T}}, 2\frac{1+T}{1+2T}\sin{\frac{\pi}{2+2T}} \right]\right\}, \\
\Gamma_2^{+}=\left\{(x,y):\,x=t,\,y=\frac{(1+T)t-1}{1+2T},\,t\in\left[0,\frac4{2+3T}\right]\right\},\\
\Gamma_2^{-}=\left\{(x,y):\,x=-t,\,y=\frac{(1+T)t-1}{1+2T},\,t\in\left[0,\frac4{2+3T}\right]\right\},\\
\Gamma_3^{+}=\left\{(x,y):\,x=A(t),\,y=B(t),\,t\in\left[0,\frac{\pi}{2+2T}\right]\right\},\\
\Gamma_3^{-}=\left\{(x,y):\,x=-A(t),\,y=B(t),\,t\in\left[0,\frac{\pi}{2+2T}\right]\right\},
\end{gather*}
where 
\begin{gather*}
A(t)=\frac{2T\sin{(2+2T)t}-(2+2T)\sin{2Tt}}{T\sin{(2+3T)t-(2+3T)\sin{Tt}}},\\
B(t)=\frac{T\sin{(2+T)t}-(2+T)\sin{Tt}}{T\sin{(2+3T)t-(2+3T)\sin{Tt}}}.
\end{gather*}

The boundaries of the intervals of variation for the parameters in defining the curves, which determine the boundary of
the domain $U_T$, are computed as the roots of the corresponding equations. Note that the domain $U_1$ is defined 
in \cite{Bran,Cow}, and optimization for $U_1$ is used in \cite{Dmitrishin2019UnivalentPA}. Therefore, the current
work is a generalization of \cite{Dmitrishin2019UnivalentPA}.

To demonstrate how the domain $U_T$ contracts with the growth of $T$, let us picture this area for different $T$ 
(Fig.~\ref{f1}). In Fig.~\ref{f2}, the boundaries $\Gamma_1$, $\Gamma_2^{+}$, $\Gamma_2^{-}$,
$\Gamma_3^{+}$, $\Gamma_3^{-}$ are shown for $T = 4$. The trinomial from family \eqref{Ts},
\begin{equation}\label{tri}
S^{(T)}(z)=z+a^{(0)}z^{1+T}+b^{(0)}z^{1+2T}
\end{equation}
where $a^{(0)}=2\frac{1+T}{1+2T}\sin{\frac{\pi}{2+2T}}$ and $b^{(0)}=\frac1{1+2T}$, has coefficients given by the point $C$ where $\Gamma_1$ and $\Gamma_3^+$ meet.

\begin{figure}[h!]
\includegraphics[scale=0.5]{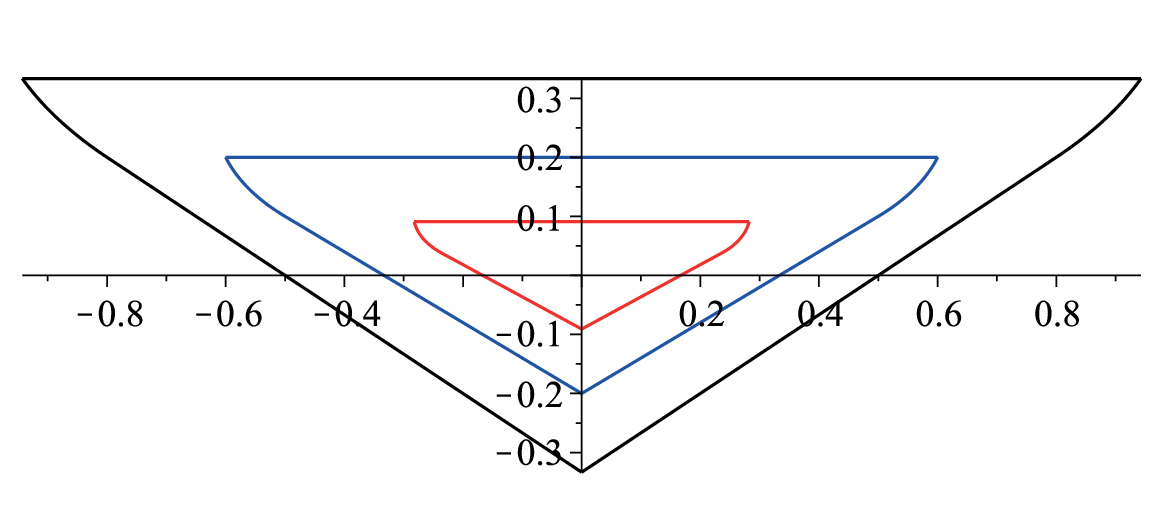}
\caption{Domain $U_T$ of univalence of the trinomial $F(z)=z+az^{1+T}+bz^{1+2T}$ in the parameter plane $(a,b)$
for $T=1$ (the boundary is shown in black), $T=2$ (the boundary is shown in blue),
$T=5$ (the boundary is shown in red)} \label{f1}
\end{figure} 

\begin{figure}[h!]
\includegraphics[scale=0.5]{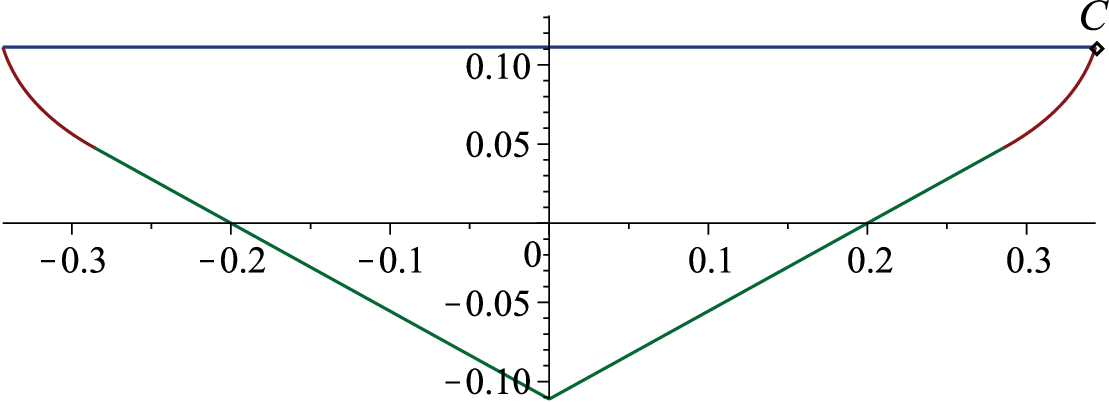}
\caption{Curves $\Gamma_1$ (blue), $\Gamma_2^{+}$, $\Gamma_2^{-}$ (green), 
$\Gamma_3^{+}$, $\Gamma_3^{-}$ (red) and the point $C(a^{(0)}, b^{(0)})$ for $T=4$} \label{f2}
\end{figure} 

Note some properties of the domain $U_T$. This area, considered in the plane $(a,b)$, has axial symmetry about the
line $Ob$; the line $\Gamma_1$ defines polynomials with zeros of the derivative lying on the unit circle (quasi-extremal
polynomials in the sense of Ruscheweyh). Thus, hypotheses {\it a)} and {\it b)} for trinomial~\eqref{tri} are confirmed. 

To confirm the truth of hypotheses {\it c)}, {\it d)}, {\it e)}, we need to determine the extreme values of the functions
$$
L_1(x,y)=x,\; L_2(x,y)=1+x+y,\; L_3(x,y)=\frac{1-x+y}{1+x+y}
$$
in the domain $U_T$. 

It is easy to show that these functions do not have stationary points in $U_T$ (points at which both partial derivatives disappear) therefore we can restrict our attention to the boundary of the region. 
 It will be shown that the functions $L_1(x,y)$ and $L_2(x,y)$ increase on the curve
 $\Gamma_3^{+}$, and the function $L_3(x,y)$ decreases on this curve. This will imply that
\begin{gather*}
\max_{(x,y)\in U_T}\{L_1(x,y)\}=L_1(a^{(0)},b^{(0)})=a^{(0)}=2\frac{1+T}{1+2T}\sin{\frac{\pi}{2+2T}},\\
\max_{(x,y)\in U_T}\{L_2(x,y)\}=L_2(a^{(0)},b^{(0)})=1+a^{(0)}+b^{(0)}=\\
=1 + \frac1{1+2T}\left( 1+2(1+T)\sin{\frac{\pi}{2+2T}} \right),\\
\min_{(x,y)\in U_T}\{L_3(x,y)\}=L_3(a^{(0)},b^{(0)})=\frac{1-a^{(0)}+b^{(0)}}{1+a^{(0)}+b^{(0)}}=
\left(\frac{1-\sin{\frac{\pi}{2+2T}}}{\cos{\frac{\pi}{2+2T}}}\right)^2.
\end{gather*}

Though the statement of the extremal problems is rather simple, we will need to overcome certain technical difficulties
to solve them. 

\section{Auxiliary results}

\begin{lemma}\label{l1}
Let 
$$
H(x,k)=-(1-x)^2 3^{2k}+(1+x)(3-2x)^{2k}-(1+x)(3-x)+(3-x)(1-2x)^{2k}.
$$
Then, for $k\ge3$ and $x\in(0,1/2]$
\begin{equation}\label{a)} 
H(x,k)<0,
\end{equation}
and
\begin{equation}\label{b)} 
\abs{H(x,k)} >\frac{\pi^2}{4(2k+1)(2k+2)}\abs{H(x,k+1)}.
\end{equation}
\end{lemma}
\begin{proof}
Let us start with the proof of \eqref{a)}.
It is obvious that $(1+x)(3-x)>(3-x)(1-2x)^{2k},$ i.e. $-(1+x)(3-x)+(3-x)(1-2x)^{2k}<0.$ Further,
$$
(1-x)^2 3^{2k}-(1+x)(3-2x)^{2k}=3^{2k}\left((1-x)^2-(1+x)(1-2x/3)^{2k}\right).
$$
For $k\ge3$, $x\in(0,1/2]$, the following inequality holds:
$$
(1-x)^2-(1+x)(1-2x/3)^{2k} \geq (1-x)^2-(1+x)(1-2x/3)^6.
$$
It suffices to show that $(1 - x)^2 - (1 + x)(1 - 2x/3)^6 > 0$. Change the inequality to the following form:
\begin{displaymath}
\frac{(1 - x)^2}{1 + x} > \left(1 - \frac{2x}{3}\right)^6.
\end{displaymath} 
For $x \in (0,\frac{1}{2}]$, note that $\frac{1}{1 + x} > 1 - x$. This implies $\frac{(1 - x)^2}{1 + x} > (1 - x)^3$. Hence, we will show that
\begin{displaymath}
(1 - x)^3 > \left(1 - \frac{2x}{3}\right)^6.
\end{displaymath}
For $x \in (0,\frac{1}{2}]$, we have
\begin{displaymath}
(1 - \frac{2x}{3})^2 = 1 - \frac{4x}{3} + \frac{4x^2}{9} = 1 - x - \frac{x(3 - 4x)}{9} < 1- x,
\end{displaymath}
which verifies $(1 - x)^3 > (1 - \frac{2x}{3})^6$. Thus, \eqref{a)} is proven.\\

Now, we proceed to the proof of inequality \eqref{b)}. Since $4(2k + 1)(2k + 2) \geq 224$ when $k \geq 3$, it suffices to show
$$
\abs{H(x,k)}>\frac{\pi^2}{224}\abs{H(x,k+1)}.
$$
Then, we must demonstrate that
\begin{eqnarray*}
\lefteqn{(1 - x)^23^{2k} - (1 + x)(3 - 2x)^2k + (1 + x)(3 - x) - (3 - x)(1 - 2x)^{2k}} \\
&>& \frac{\pi^2}{224}\Big(9(1 - x)^23^{2k} - (3 - 2x)^2(1 + x)(3 - 2x)^{2k} \\
& & + (1 + x)(3 - x) - (1 - 2x)^2(3 - x)(1 - 2x)^{2k}\Big).
\end{eqnarray*}
First, we will show that 
$$
(1+x)(3-x)-(3-x)(1-2x)^{2k}>\frac{\pi^2}{224}\left((1+x)(3-x)-(3-x)(1-2x)^{2k}(1-2x)^2\right).
$$
Simplify the inequality to 
$$
1+x>(1-2x)^{2k}\left(\frac{1 - \frac{\pi^2}{224}(1-2x)^2}{1 - \frac{\pi^2}{224}}\right).
$$
We can further simplify this expression by noting
\begin{displaymath}
\frac{1 - \frac{\pi^2}{224}(1-2x)^2}{1 - \frac{\pi^2}{224}} = 1 + \frac{4\pi^2x(1 - x)}{224 - \pi^2}.
\end{displaymath}
Since $x(1 - x) < x$ and $(1 - 2x)^{2k}$ is decreasing for $x \in (0,\frac{1}{2}]$, we have
\begin{displaymath}
1+x>(1-2x)^{2k}\left(1 + \frac{4\pi^2x}{224 - \pi^2}\right) > (1 - 2x)^{2k}\left(1 + \frac{4\pi^2x(1 - x)}{224 - \pi^2}\right).
\end{displaymath}
This proves the inequality.

Now, let us show that
$$
(1-x)^23^{2k}-(1+x)(3-2x)^{2k}>\frac{\pi^2}{224}\left(9(1-x)^23^{2k}-(3 - 2x)^2(1+x)(3-2x)^{2k}\right).
$$
To do that let us write the inequality in the form
$$
(1 - x)^2 > (1 + x)\left(1 - \frac{2}{3}x\right)^{2k}\left(\frac{224 - \pi^2(3 - 2x)^2}{224 - 9\pi^2}\right).
$$
Since $(1 - 2x/3)^{2k} \leq (1 - 2x/3)^6$ for $k\ge3$ and $x\in(0,1/2]$ it suffices to show
$$
(1 - x)^2 > (1 + x)\left(1 - \frac{2}{3}x\right)^{6}\left(\frac{224 - \pi^2(3 - 2x)^2}{224 - 9\pi^2}\right).
$$
For $x \in (0,\frac{1}{2}]$, we have
\begin{displaymath}
\frac{224 - \pi^2(3 - 2x)^2}{224 - 9\pi^2} = 1 + \frac{4\pi^2x(3 - x)}{224 - 9\pi^2} < 1 + \frac{12\pi^2x}{224 - 9\pi^2} < 1 + x.
\end{displaymath}
Hence, the problem reduces to demonstrating that
\begin{displaymath}
(1 - x)^2 > (1 + x)^2\left(1 - \frac{2}{3}x\right)^6,
\end{displaymath}
or equivalently
\begin{displaymath}
(1 - x)^2 - (1 + x)^2\left(1 - \frac{2}{x}x\right)^6 > 0.
\end{displaymath}
With a little algebra we obtain
\begin{eqnarray*}
\lefteqn{(1 - x)^2 - (1 + x)^2\left(1 - \frac{2}{x}x\right)^6} \\
&=& \frac{4x^2}{729}(4x^2 - 14x + 9)(27 - 27x - (4x^2 - 14x + 9)) \\
&>& \frac{4x^2}{729}(9 - 14x)(27 - 27x - (4x^2 - 14x + 9)) \\
&=& \frac{4x^2}{729}(9 - 14x)(27(1 - x) - (-x + (1 - x)(9 - 4x)) \\
&=& \frac{4x^2}{729}(9 - 14x)(x + (1 - x)(18 + 4x)) \\
&>& 0,
\end{eqnarray*}
for $x \in (0,\frac{1}{2}]$. Thus, the inequality \eqref{b)} is valid.
 
The lemma is proved.
\end{proof}

\begin{lemma}\label{l2}
Let 
\begin{gather*}
G(\tau,\alpha)=(1+\alpha)\cos{(3-2\alpha)\tau}+(3-\alpha)\cos{(1-2\alpha)\tau}-\\
-(1-\alpha)^2\cos{3\tau}-(1+\alpha)(3-\alpha)\cos{\tau}.
\end{gather*}
Then $G(\tau,\alpha) \geq 0$ when $\tau \in (0,\pi/2)$, $\alpha\in(0,1/2]$.
\end{lemma}
\begin{proof}
Using the cosine series expansion formula, we can write
\begin{equation}\label{G}
G(\tau,\alpha)=\sum_{k=0}^{\infty} (-1)^k \frac1{(2k)!}H(\alpha,k)\tau^{2k}.
\end{equation}
Note that $H(\alpha,0)=H(\alpha,1)=H(\alpha,2)\equiv0$. By \eqref{a)} the series~\eqref{G} is an alternating
series with $H(\alpha,3)<0$, and by \eqref{b)} the series is monotonically decreasing in absolute value for $k \geq 3$ and $\tau \in (0,\frac{\pi}{2}]$. Using Leibniz's Theorem on estimating a function with the partial sums of its alternating 
series expansion, we obtain the estimate
$$
G(\tau,\alpha)>\frac{\tau^6}{8!}\left(56 \cdot \abs{H(\alpha,3)}-\abs{H(\alpha,4)}\tau^2\right)=
\frac{\abs{H(\alpha,4)}\tau^6}{8!} \left(\frac{14}{5-4\alpha+2\alpha^2}-\tau^2\right).
$$
When $\alpha\in(0,1/2]$, $\displaystyle\frac{14}{5-4\alpha+2\alpha^2}>\frac{14}5$. Then
$$
G(\tau,\alpha)>\frac{\abs{H(\alpha,4)}\tau^6}{8!} \left(\frac{14}{5}-\left(\frac{\pi}2\right)^2\right)>0
$$
for $\tau\in(0,\pi/2)$. The lemma is proved.

\end{proof}
\section{Main results}

Let us find the extreme values of the functions 
$$
L_1(x,y)=x,\; L_2(x,y)=1+x+y,\; L_3(x,y)=\frac{1-x+y}{1+x+y}
$$ 
in the domain $U_T$.

\begin{theorem}\label{t1}
The functions $L_1(x,y)$ and $L_2(x,y)$ attain the maximum over the region $U_T$ at the upper right corner, while the function $ L_3(x,y)$ attains its minimum there. 
\end{theorem}

\begin{proof}
Because functions $L_j(x,y)$  do not have stationary points in $U_T$ and  due to the symmetry of the domain $U_T$, we can restrict ourselves to the boundary and to the case $x\ge0$. Obviously, the functions $L_1(x,y)$ and $L_2(x,y)$ increase on the curves $\Gamma_1$ and $\Gamma_2^{+}$.
And the function $L_3(x,y)$ decreases on these curves. It remains to consider the behavior of these functions on the
curve $\Gamma_3^{+}$. 

Since $\frac{1-x+y}{1+x+y}=1-\frac2{1+\frac{1+y}x}$, instead of the initial functions
$L_2(x,y)$, $L_3(x,y)$ it is more convenient to consider the functions $\hat{L}_2(x,y)=y$, $\hat{L}_3(x,y)=\frac{1+y}x$. 

We represent the parametric equations defining the curve $\Gamma_3^{+}$ in the form
$$
\Gamma_3^{+}=\left\{(x,y):\,x=\frac{U(t)}{W(t)},\,y=\frac{V(t)}{W(t)},\,t\in\left[0,\frac{\pi}{2+2T}\right]\right\},
$$ 
where
\begin{gather*}
U(t)=2T\sin{(2+2T)t}-(2+2T)\sin{2Tt},\\
V(t)=T\sin{(2+T)t}-(2+T)\sin{Tt},\\
W(t)=T\sin{(2+3T)t}-(2+3T)\sin{Tt}.
\end{gather*}
Make the substitution $\tau=(1+T)t$, $\alpha=\frac1{1+T}$. Then $\tau\in[0,\pi/2]$, $\alpha\in(0,1/2]$. The functions
$U(t)$, $V(t)$, $W(t)$ transform into the functions 
\begin{gather*}
\hat{U}(\tau)=\frac1{\alpha}(2(1-\alpha)\sin{2\tau} - 2\sin{(2(1-\alpha))\tau}),\\
\hat{V}(\tau)=\frac1{\alpha}((1-\alpha)\sin{(1+\alpha)\tau} - (1+\alpha)\sin{(1-\alpha)\tau}),\\
\hat{W}(\tau)=\frac1{\alpha}((1-\alpha)\sin{(3-\alpha)\tau} - (3-\alpha)\sin{(1-\alpha)\tau}),
\end{gather*}
and
$$
\Gamma_3^{+}=\left\{(x,y):\,x=\frac{\hat{U}(\tau)}{\hat{W}(\tau)},
\,y=\frac{\hat{V}(\tau)}{\hat{W}(\tau)},\,\tau\in\left[0,\frac{\pi}2\right]\right\}.
$$

Find the derivatives
\begin{gather*}
\hat{U}'(\tau)=\frac8{\alpha}(1-\alpha)\sin{\alpha\tau}\sin{(2-\alpha)\tau},\\
\hat{V}'(\tau)=\frac2{\alpha}(1-\alpha)(1+\alpha)\sin{\alpha\tau}\sin{\tau},\\
\hat{W}'(\tau)=\frac2{\alpha}(3-\alpha)(1-\alpha)\sin{(2-\alpha)\tau}\sin{\tau}.
\end{gather*}
Then

\begin{gather*}
\left(\frac{\hat{U}(\tau)}{\hat{W}(\tau)}\right)'=
\frac{2(1-\alpha)\sin{(2-\alpha)\tau}}{\alpha^2\left(\hat{W}(\tau)\right)^2}G(\tau,\alpha),\\
\left(\frac{\hat{V}(\tau)}{\hat{W}(\tau)}\right)'=
\frac{2(1-\alpha)\sin{\tau}}{\alpha^2\left(\hat{W}(\tau)\right)^2}G(\tau,\alpha),\\
\left(\frac{\hat{V}(\tau)}{\hat{U}(\tau)}\right)'=
\frac{2(1-\alpha)\sin{\alpha\tau}}{\alpha^2\left(\hat{U}(\tau)\right)^2}G(\tau,\alpha),\\
\left(\frac{\hat{W}(\tau)}{\hat{U}(\tau)}\right)'=
-\frac{2(1-\alpha)\sin{(2-\alpha)\tau}}{\alpha^2\left(\hat{U}(\tau)\right)^2}G(\tau,\alpha),
\end{gather*}
and
$$
\left(\frac{\hat{V}(\tau)+\hat{W}(\tau)}{\hat{U}(\tau)}\right)'=
-\frac4{\alpha^2\left(\hat{U}(\tau)\right)^2}(1-\alpha)\cos{\tau}\sin{(1-\alpha)\tau}G(\tau,\alpha).
$$

By Lemma~\ref{l2}, $G(\tau,\alpha)>0$ when $\tau\in(0,\pi/2)$, $\alpha\in(0,1/2]$. Therefore, 
$$
\left(\frac{\hat{U}(\tau)}{\hat{W}(\tau)}\right)'>0,\;
\left(\frac{\hat{V}(\tau)}{\hat{W}(\tau)}\right)'>0,\;
\left(\frac{\hat{V}(\tau)+\hat{W}(\tau)}{\hat{U}(\tau)}\right)'<0.
$$
Hence, we proved that on the curve $\Gamma_3^{+}$:

1) the function $L_1(x,y)=x$ increases; 

2) the function $\hat{L}_2(x,y)=y$ increases;

3) the function $\hat{L}_3(x,y)=\frac{1+y}x$ decreases.

The theorem is proved. 
\end{proof}

It follows from the Theorem that hypotheses {\it c), d), e)} are true. 

Indeed, the point $(x,y)$ in the domain corresponds to a univalent polynomial $z+xz^{1+T}+yz^{1+2T}.$ In particular, among univalent
polynomials $\displaystyle z+xz^{1+T}+\frac1{1+2T}z^{1+2T}$ the polynomial $S^{(T)}(z)$ has the largest coefficients, the corresponding point $(x,y)$ is at the upper right corner of the region. This proves conjecture c).

Further, due to symmetry of the region and situation in an upper half-plane we might assume both $x$ and $y$ to be positive in consideration of conjecture d). In that case the maximum of the polynomial is attained at $z=1,$ i.e.
$$
\max_{z\in\bar{\mathbb D}}\left\{ z+xz^{1+T}+yz^{1+2T}\right\}=1+x+y.
$$
The maximum of $1+x+y$ is achieved at the upper right corner of the domain, hence the value $\abs{S^{(T)}(1)}$ is the largest one, which proves conjecture d).

Conjecture e) follows by similar argument.

The extremal polynomial that corresponds to
hypothesis {\it e)} is unique: $S^{(T)}(z)$; in the problems corresponding to hypotheses {\it c)} and {\it d)}, these are
the polynomials $S^{(T)}(z)$, $-S^{(T)}(-z)$.

\section{Discussion of the results}

Of course, the next problem in the queue for consideration is the case of polynomials of degree $N=3T + 1,$ or more 
generally, $N = (n - 1)T + 1$ for $n\ge 4.$ The domains of univalence in these cases are not known even for $T=1$ and $n \geq 5$. This does not prevent considering hypotheses {\it a)-e)} in the general situation, but greatly complicates the solution. 

The next natural question is the question about the maximum contraction of the unit disk in the direction of the real axis
for univalent polynomials with the schlicht normalization $F(0)=0$, $F'(0)=1$. Note that in this case the Suffridge
polynomials are not optimal \cite{Dmitrishin2019UnivalentPA,Br1}. Thus, changing the normalization $F(0)=0$, $F'(0)=1$
to $F(0)=0$, $F(1)=1$ leads not only to a change of the extreme value, but also to a change of the extremal polynomial,
which is quite interesting.

In \cite{DST}, there were introduced the polynomials $s^{(T)}(z)$ that formally differ from the polynomials $S^{(T)}(z)$
defined by \eqref{Ts}. The following asymptotics was established:
$$
\frac1{s^{(T)}(1)}\abs{s^{(T)}\left(e^{i\frac{\pi}T}\right)} \sim c_T n^{-2/T}
$$
as $n\rightarrow\infty$, where $c_T=\pi^{2/T-1}\Gamma^2(1/T+1/2)$, $\Gamma(\cdot)$ is the gamma function.

However, already \cite{DSST} introduces and proves, for special cases, the conjecture that the polynomials $s^{(T)}(z)$
coincide with the polynomials $S^{(T)}(z)$. In the same work, the following relation was shown:
$$
\max\limits_{z\in\overline{\bD}}\left\{\abs{s^{(T)}(z)}\right\}=s^{(T)}(1)\sim\frac1{c_T 2^{2/T}}n^{2/T}
$$ 
as $n\rightarrow\infty,$ which is a quantitative refinement of hypothesis {\it d)}. 
 

\section{Acknowledgment}

The authors would also like to thank Larie Ward for her help in preparation of the manuscript.

\bibliographystyle{unsrt}
\bibliography{tribib}

\end{document}